\documentclass[11pt]{amsart}
\usepackage{latexsym}
\usepackage{rotating}
\usepackage{amsfonts}
\usepackage{amssymb,amscd}
\usepackage{amsmath}
\usepackage{hyperref}
\newtheorem{theorem}{Theorem}[section]
\newtheorem{lemma}[theorem]{Lemma}     
\newtheorem{cor}[theorem]{Corollary}

\newtheorem{definition}[theorem]{Definition}

\usepackage[margin=3.8 cm]{geometry}

\newcommand{\secref}[1]{Section~\ref{#1}}
\newcommand{\thmref}[1]{Theorem~\ref{#1}}
\newcommand{\lemref}[1]{Lemma~\ref{#1}}

\newcommand{\corref}[1]{Corollary~\ref{#1}}

\numberwithin{equation}{section}
\newtheorem*{theorema}{Theorem JKK}

\newtheorem*{definitiona}{Definition}

\begin{document}

\title[$z$-Classes in finite groups of conjugate type $(n,1)$ ]{$z$-Classes in finite groups of conjugate type $(n,1)$}

\author[S. Arora]{Shivam Arora}
\author[K. Gongopadhyay]{Krishnendu Gongopadhyay}

\address{Indian Institute of Science Education and Research (IISER) Mohali, Knowledge City, ,
Sector 81,  S.A.S. Nagar, Punjab 140306, India}
\email{shivam.iiserm@gmail.com}

\address{Indian Institute of Science Education and Research (IISER) Mohali, Knowledge City, 
Sector 81,  S.A.S. Nagar, Punjab 140306, India}

\email{krishnendug@gmail.com, krishnendu@iisermohali.ac.in}

\begin{abstract}
Two elements in a group $G$ are said to \emph{$z$-equivalent} or to be in the same \emph{$z$-class} if their centralizers are conjugate in $G$. In \cite{kkj}, it was proved that a non-abelian $p$-group $G$ can have at most $\frac{p^k-1}{p-1} +1$ number of $z$-classes, where $|G/Z(G)|=p^k$. In this note, we characterize the $p$-groups of conjugate type $(n,1)$ attaining this maximal number. As a corollary, we characterize $p$-groups having prime order commutator subgroup and maximal number of $z$-classes. 
\end{abstract}
\maketitle
\section{Introduction} 
Two elements in a group $G$ are said to \emph{$z$-equivalent} or to be in the same \emph{$z$-class} if their centralizers are conjugate in $G$. The $z$-equivalence is an equivalence relation in $G$ and give a partition of $G$ into disjoint subsets. It is coarser than the conjugacy relation. An infinite group may have infinitely many conjugacy classes, but often the number of $z$-classes is finite and this finiteness of $z$-classes give a rough idea about dynamical types in a homogeneous space on which $G$ acts, see \cite{kul1}. There have been some recent work where $z$-classes in various groups have been investigated, see \cite{rony, kulaj, singh, gk1, krish, gk2}.  We note here that a closely related problem is to compute the \emph{genus number} that has been investigated in \cite{bose}. Jadhav and Kitture has classified $z$-classses in $p$-groups of order up to $p^5$, \cite{jaki}. 

In a recent work Jadhav, Kitture and Kulkarni \cite{kkj} has investigated $z$-classes in finite $p$-groups. They proved that a non-abelian $p$-group has at least $p+2$ $z$-classes and characterized the groups that attains this lower bound. They also proved that a non-abelian $p$-group can have at most $\frac{p^k-1}{p-1}+1$ number of $z$-classes, where $|G/Z(G)|=p^k$. They proved the following theorem that gives necessary conditions for a $p$-group $G$ to attain this maximal number of $z$-classes. 

\begin{theorema}\cite{kkj}
	Let $G$ be a non-abelian group with $|G/Z(G)|=p^k$. If the number of z-classes in $G$ is $\frac{p^k-1}{p-1} +1 $ then either $G/Z(G) \cong C_p \times C_p$ or the following holds:\\
	(1) $G$ has no abelian subgroup of index $p$.\\
	(2) $G/Z(G)$ is an elementary abelian $p$-group.
\end{theorema}
The authors further gave a counterexample to show that conditions (1) and (2) are not sufficient to attain this bound. It would be interesting to obtain sufficient conditions for a group $G$ to attain the maximal number of $z$-classes and to characterize such groups.   

\medskip In this note, we characterize groups of type $(n,1)$ attaining this bound. To state our result, we recall the following definition by Ito \cite{ito1}. 

\begin{definitiona} \cite{ito1} Let $G$ be a finite group. Let $n_1,n_2,\ldots, n_r$ , where $n_1>n_2>\ldots>n_r=1$, be all the numbers which are the indices of the centralizers of elements of $G$ in $G$. The vector $(n_1,n_2,\ldots,n_r)$ is called conjugate type vector of $G$. A group with the conjugate type vector $(n_1,n_2,\ldots, n_r)$ is called a group of type $(n_1,n_2,\ldots, n_r)$.
\end{definitiona}

In \cite{ito1}, Ito characterizes the groups of type $(n,1)$. Any such group is nilpotent and $n=p^a$ for some prime $p$. A group of type $(p^a,1)$ is the direct product of a $p$-group of the same type and of an abelian group. Thus,  the study of groups of type $(n,1)$ can be reduced to that of $p$-groups of type $(p^a,1)$.  Ito showed that such a group $G$ contains an abelian normal subgroup $A$ such that $G/A$ is of exponent $p$. 

\medskip Here we prove the following. 
\begin{theorem}\label{mt}
	Let $G$ be a non-abelian group of type $(n,1)$ and $[G:Z(G)]=p^k$. Then $G$ has $\frac{p^k-1}{p-1} +1 $  $z$-classes if and only if
\begin{enumerate}
	\item $G/Z(G)$ is elementary abelian, and 
	\item for all $x\in G\setminus Z(G)$, $Z(C_G(x))=\langle x, Z(G) \rangle$.
\end{enumerate}
\end{theorem}  

As an application of the above result we have the following. 
\begin{cor}\label{est}
	Let $G$ be a non-abelian group  with $|G/Z(G)|=p^k$, $(k\geq 2)$ and $|G'|=p$. Then the number of z-classes in $G$ is $\frac{p^k-1}{p-1} +1 $ if and only if  G is isoclinic to an extraspecial $p$-group. 
\end{cor} 
We prove these results in \secref{main}. 

\medskip We note that  \corref{est} may be compared with Theorem 5.3.1 of Kitture \cite[Section 5.3]{rt}. We thank Kitture for letting us know about his result and many comments. We are also thankful to Silvio Dolfi for useful comments and suggestions on a first draft of this note. 
\section{Preliminaries}
We recall here a few basic facts from Kulkarni \cite{kul1} that will be used in the proof. 

Let $G$ be a group. We denote $C_G(x)$ to be the centralizer of an element $x$ in $G$. 
\begin{definition}
	Let $G$ be a group. Two elements $x,y\in G$ are said to be $z$-equivalent if their centralizers are conjugate, i.e. there exist $g\in G$ such that,
	$$ g^{-1} C_G(x) g=C_G(y).$$
\end{definition}

Let $G$ be a group acting on a set $X$. For $x\in X$ let $G(x)$ denote the orbit of $x$ and $G_x$ be the stabilizer subgroup of $x$: 
$$G(x)=\{y\in X : \hbox{there exists } g\in G\ and\ y=gx\}.$$ 
$$G_x=\{g\in G : gx=x\}.$$
We have a partition of $X$ as: 
$$ X= \bigcup \limits_{x\in X} G(x).$$
 Let $y\in G(x)$ be some element, then for some $g\in G$ we have $y=gx$ and $G_y=gG(x)g^{-1}$. 

For subgroups $A$ and $B$ in $G$ write $A\sim_{g}B$ if they are conjugate. This gives us the following. 
\begin{definition}
	Let $x,y\in X$ be any two element, they are said to be in same orbit class if $G_x \sim_g G_y$. Denote this equivalence relation as $x\sim_0 y$ and denote $R(x)$ as the equivalence class of $x$.
\end{definition}
 This gives another partition of $X$:
$$X  = \bigcup \limits_{G_x \sim_g G_y} R(x)=\bigcup \limits_{G_x = G_y} R(x).$$
The second inequality is easy to note, since if $x\sim_0 y$ then for some $z\in X$, we have $C_G(x)=C_G(z)$.

\medskip 
Let $F_x$ denotes the points fixed by $G_x$ i.e. 
$$F_x=\{y\in X : G_y \supset G_x\}.$$ 
Let $F'_x$ denotes the set $F'_x=\{y\in X : G_y=G_x\}.$ Define $W_x= N_x/G_x$, where $N_x$ is the normalizer of the $G_x$ in $G$. Then we have the following theorem(refer \cite{kul1}).
\begin{theorem}
	
	Let $G$ act on a set $X$. Then the map $\phi : G/G_x \times F'_x \rightarrow R(x)$, 
	$ \phi(gG_x,y)=gy$
	is well defined and induces a bijection,
	$\bar{\phi}:(G/G_x \times F'_x)/W_x \rightarrow R(x).$
	
\end{theorem}
\noindent When we take our space $X=G$, the group itself, then $R(x)$ in above notation becomes the $z$-classes and the above theorem gives us the following equation, which is used to calculate the size of a $z$-class: for any $x\in G$ we have 
$$|z-class\ of\ x|= [G:N_G(C_G(x))].|F'_x|,$$
where $C_G(x)$ is the centralizer of $x$ and $F'_x=\{y\in G : C_G(y)=C_G(x) \}$.

\medskip 
 The following theorem states that that number of $z$-classes is invariant under the isoclinism. For basic notions on isoclinism, see \cite{hall, kkj}. 
\begin{theorem} \cite{kkj}
	Let $G_1$ and $G_2$ be isoclinic groups, with an isoclinism $(\phi, \psi)$. Then the isoclinism $(\phi, \psi)$ induces a bijection between the $z$-classes in $G_1$ and $G_2$.
\end{theorem}
Finally we note the following theorem by Hall \cite{hall} that is used in proof of \thmref{est}. 
\begin{theorem}\label{hall}
	Every group is isoclinic to a group whose center is contained in the commutator subgroup. 
\end{theorem}

\section{Proof of The Main Results}\label{main}
It was proved in \cite[Lemma 3.1]{kkj} that if $G$ is a finite group with $|G/Z(G)|=p^k$, then $G$ is isoclinic to a finite $p$-group. Since number of $z$-classes is invariant under isoclinism, hence it is enough to prove our results assuming that $G$ is a $p$-group. 
\begin{lemma}\label{lem1}
Let $G$ be a non-abelian $p$-group with $[G: Z(G)]=p^k$. If 
\begin{enumerate}
\item $G/Z(G)$ is elementary abelian, and
\item for all $z \in G \setminus Z(G)$, $Z(C_G(x))=\langle x, Z(G)\rangle$, 
\end{enumerate}
then $G$ has $\frac{p^k-1}{p-1} +1$ $z$-classes. 
\end{lemma}
\begin{proof}
Let $x\in G\setminus Z(G)$. Since $G/Z(G)$ is abelian, hence $C_G(x) \unlhd G$, i.e.  $[G:N(C_G(x)]=1$. Consider
 $$F'_x=\{y\in G : C_G(y)=C_G(x) \}.$$
	By definition $F'_x \subseteq Z(C_G(x)) \setminus Z(G)=\langle x,Z(G)\rangle \setminus Z(G) $ . Obviously $C_G(x)=C_G(xt)$ for $t\in Z(G)$. Also since exponent of $G/Z(G)$ is $p$, thus $C_G(x)=C_G(x^i)$ for any $i\in \{1,2,..,p-1\}$. Hence $\langle x,Z(G) \rangle \setminus Z(G) \subseteq F'_x$. Hence $$F'_x = Z(C_G(x)) \setminus Z(G) =\langle x,Z(G)\rangle \setminus Z(G)= xZ(G) \cup ...\cup x^{p-1}Z(G). $$
This implies,  $|F'_x|=(p-1)|Z(G)|$. Therefore, 
	$$|z-class\ of\ x|=[G:N_G(C_G(x))].|F'_x|= 1.(p-1)|Z(G)|.$$ 
Thus $G\setminus Z(G)$ contributes to a total $\frac{p^k-1}{p-1}$ number of $z$-classes of $G$. Clearly, $Z(G)$ is also a $z$-class. Hence,  $G$ has $\frac{p^k-1}{p-1} +1 $ $z$-classes. 
\end{proof}  

\begin{cor}
	Let $G$ be a non-abelian $p$-group with $|G/Z(G)|=p^k$. If
\begin{enumerate}
	\item  $G$ has no abelian subgroup of order exceeding $p |Z(G)|.$
	\item $G/Z(G)$ is an elementary abelian $p$-group.\end{enumerate}
	Then $G$ has $\frac{p^k-1}{p-1} +1 $ number of z-classes.
\end{cor}
\begin{proof}We claim 
	 $C_G(x)=\langle x, Z(G) \rangle$, and the rest follow from \lemref{lem1}. 

 Clearly $\langle x,Z(G) \rangle \subseteq C_G(x)$. Let us assume $y\notin \langle x,Z(G) \rangle$ and $y\in C_G(x)$. But then $\langle x,y,Z(G) \rangle $is an abelian subgroup and $|\langle x,y,Z(G) \rangle| > p |Z(G)| $. Thus $C_G(x)=\langle x,Z(G)\rangle$.
\end{proof}

\subsection{Proof of \thmref{mt}}
\begin{proof}
	Let $G$ be an non-abelian group of type $(n,1)$ and $G$ has $\frac{p^k-1}{p-1} +1 $ $z$-classes, where $[G:Z(G)]=p^k$. 	
(1) follows from the Theorem JKK. We prove (2). 

Let $x\in G\setminus Z(G)$, then for any $t\in Z(G)$, $C_G(x)=C_G(xt)$. Also for any $m$ relatively prime to order of $x$, we have 
	$$C_G(x)=C_G(\langle x\rangle )=C_G(\langle x^m \rangle )$$
	Thus $z$-class of $x$ contains, 
	$$xZ(G) \cup x^2Z(G) \cup.....\cup x^{p-1}Z(G),$$
and hence,  for all $x\in G \setminus Z(G)$, the size of $z$ class is at least $(p-1)|Z(G)|$.  Since $Z(G)$ is also a $z$-class, hence it follows from the hypothesis that  $G\setminus Z(G)$  constitute exactly $\frac{p^k-1}{p-1}\  z$-classes. Now, $|G\setminus Z(G)|=|G|-|Z(G)|$.  Hence, we must have for $x \in G \setminus Z(G)$, 
$$|z-class \ of\  x|=(p-1)|Z(G)|.$$

Now, it is clear that for $x \in G \setminus Z(G)$,  $\langle x, Z(G) \rangle \subseteq Z(C_G(x))$. Now let $y\in Z(C_G(x))$ such that $y\notin \langle x, Z(G) \rangle $.  Then a simple observation shows that $C_G(y) \supseteq C_G(x)$.   
	But since $G$ is of type $(n,1)$ we have $|C_G(x)|=|C_G(y)|$, thus we must have $C_G(x)=C_G(y)$. This implies $y\in (z-class\ of\ x)$ and since $y\notin \langle x, Z(G) \rangle $, thus $|z-class\ of\ x|>(p-1)|Z(G)|$, which is a contradiction. Thus our assumption is incorrect and  $Z(C_G(x))=\langle x, Z(G) \rangle$ for all $x\in G\setminus Z(G)$.

\medskip The converse part follows from \lemref{lem1}.  This completes the proof. 
\end{proof}

\subsection{Proof of \corref{est}}
\begin{proof}: Without loss of generality, using \thmref{hall}, we assume up to isoclinism that, $G$ is a group such  that $Z(G)\subseteq G'$. Let $G$ has $\frac{p^k-1}{p-1} +1 $ $z$-classes. Since $|G'|=p$ and in a non-abelian $p$-group $|Z(G)|>1$, it follows that  $G'=Z(G)$; hence, 
 $|G'|=|Z(G)|=p$. Therefore $G$ and $G/Z(G)$ have the same exponent. Thus $G$ is an extra-special $p$-group.

\medskip Conversely,  let $G$ be an extraspecial $p$-group. So,  $|Z(G)|=|G'|=p$. For $x\in G\setminus Z(G)$, define the map:
	$$\phi_x : G\rightarrow G'$$
	$$ g\mapsto[x,g]$$
As $G/Z(G)$ is elementary abelian, $\phi_x$ is a homomorphism with $ker(\phi_x)=C_G(x)$ and since $|G'|=p$, the map is also surjective. Hence we get 
	$$G/C_G(x)\cong G'$$
	Thus each $C_G(x)$ is a maximal subgroup of index $p$ in $G$. This shows that extraspecial $p$-groups are of type $(n,1)$. 

Now, we claim that $Z(C_G(x))=\langle x,Z(G) \rangle$. This follows from the fact that $G/Z(G)$ can be equipped with an alternating non-degenerate form induced by the commutator, see the proof of \cite[Theorem 3.14]{craven} or \cite[p. 83]{wilson}. The claim follows from \cite[Lemma 3.11]{craven} noting that $Z(C_G(x))/Z(G)$ is the radical of co-dimension one subspace $C_G(x)/Z(G)$. 

\medskip The result now follows from \thmref{mt}. 
\end{proof}

\end{document}